\documentclass[a4paper,12pt]{amsart}
\usepackage{amsmath,amsthm,amssymb,a4wide}
\usepackage[arrow,matrix]{xy}
\usepackage{a4wide,hyperref}
\hypersetup{
    pdftitle=   {Classification of Bott manifolds up to dimension eight},
    pdfauthor=  {Suyoung Choi}
}
\def\arxiv#1{\href{http://arXiv.org/abs/#1}{arXiv:#1}}
\def\MR#1{\href{http://www.ams.org/mathscinet-getitem?mr=#1}{MR #1}}

\theoremstyle{plain}
\newtheorem{theorem}{Theorem}[section]
\newtheorem{proposition}[theorem]{Proposition}
\newtheorem{lemma}[theorem]{Lemma}

\newtheorem{problem}[theorem]{Problem}
\newtheorem{conjecture}[theorem]{Conjecture}

\newtheorem*{SCR_three}{Theorem A}
\newtheorem*{CR_four}{Theorem B}

\theoremstyle{definition}

\newtheorem{example}[theorem]{Example}
\newcommand{\Z}{\mathbb{Z}}

\newcommand{\Q}{\mathbb{Q}}

\newcommand{\C}{\mathbb{C}}
\newcommand{\CP}{\mathbb{C}P}

\newcommand{\uC}{\underline{\mathbb{C}}}

\def\e{\epsilon}
\def\ve{\varepsilon}

\begin{document}
\title{Classification of Bott manifolds up to dimension eight}

\author[S.Choi]{Suyoung Choi}
\address{Department of Mathematics, Ajou University, San 5, Woncheon-dong, Yeongtong-gu, Suwon 443-749, Korea}
\email{schoi@ajou.ac.kr}

\date{\today}
\maketitle

%\TABLEOFCONTENTS
\begin{abstract}
We show that three- and four-stage Bott manifolds are classified up to diffeomorphism by their integral cohomology rings. In addition, any cohomology ring isomorphism between two three-stage Bott manifolds can be realized by a diffeomorphism between the Bott manifolds.
\end{abstract}

\section{Introduction}
A \emph{Bott tower} of height $n$ is a sequence of projective bundles
\begin{equation} \label{BTower}
    B_\bullet \colon B_n \stackrel{\pi_n}\longrightarrow B_{n-1} \stackrel{\pi_{n-1}}\longrightarrow \cdots \stackrel{\pi_2}\longrightarrow B_1 \stackrel{\pi_1}\longrightarrow B_0 = \{\text{a point}\},
\end{equation}
where, for $i=1, \ldots, n$, $\xi_i$ is a complex line bundle and $\uC$ is a complex line bundle over $B_{i-1}$, and $\pi_i \colon B_i = P(\uC \oplus \xi_i) \to B_{i-1}$ is a projective bundle over $B_{i-1}$. We call $B_n$ an $n$-stage Bott manifold, and $B_\bullet$ a Bott tower structure of $B_n$. Note that an $n$-stage Bott manifold is of real dimension $2n$.  A one-stage Bott manifold is the complex projective space $\CP^1$ of complex dimension one.  A two-stage Bott manifold is known as a \emph{Hirzebruch surface}. Hirzebruch \cite{hirzebruch:1951} has shown that the topological type of a Hirzebruch surface $\Sigma_a = P(\uC\oplus \gamma^{\otimes a})$ is completely determined by the parity of $a$, where $\gamma$ is the tautological line bundle over $\CP^1$; i.e., $\Sigma_a$ is homeomorphic to $\Sigma_b$ if and only if $a \equiv b (\text{mod }2)$. In addition, one can easily see that $H^\ast(\Sigma_0)$ and $H^\ast(\Sigma_1)$ are not isomorphic as graded rings. Later, it is shown that this classification also holds in the smooth category (see \cite{ma-pa:2008}), and stimulates the following conjecture (see \cite{CMS}).

\begin{conjecture}[\textbf{Cohomological rigidity conjecture for Bott manifolds}] \label{conj:CRC}
    Let $B_n$ and $B_n'$ be $n$-stage Bott manifolds. Then, $B_n$ is diffeomorphic to $B_n'$ if and only if $H^\ast(B_n)$ is isomorphic to $H^\ast(B_n')$ as graded rings.
\end{conjecture}

More strongly, we conjecture the following:
\begin{conjecture}[\textbf{Strong cohomological rigidity conjecture for Bott manifolds}] \label{conj:SCRC}
    For any cohomology ring isomorphism $\varphi$ between two Bott manifolds, there is a diffeomorphism which induces $\varphi$.
\end{conjecture}

Conjecture~\ref{conj:CRC} is known to be true for $n\leq 3$ (see \cite{ch-ma-su:2010}), and Conjecture~\ref{conj:SCRC} is known to be true for $n\leq 2$ (see \cite{ch-ma:arXiv} or Theorem~\ref{thm:SCR_Q_trivial}). However, they have been open for the higher cases. In this paper, we shall show that Conjecture~\ref{conj:SCRC} is true for three-stage Bott manifolds, and that Conjecture~\ref{conj:CRC} is true for four-stage Bott manifolds; namely, we have the following theorems.
\begin{SCR_three}[Theorem~\ref{thm:SCR_three}]
    For any cohomology ring isomorphism $\varphi$ between two three-stage Bott manifolds, there is a diffeomorphism between them, which induces $\varphi$.
\end{SCR_three}

\begin{CR_four}[Theorem~\ref{thm:CR_four}]
    Let $B_4$ and $B_4'$ be four-stage Bott manifolds. Then, $B_4$ is diffeomorphic to $B_4'$ if and only if $H^\ast(B_4)$ is isomorphic to $H^\ast(B_4')$ as graded rings.
\end{CR_four}

%We remark that the proof of Theorem~\ref{thm:SCRC_three} is quite elementary, although the proof in \cite{ch-ma-su:2010} is used the general theory. This proves that three-stage Bott manifolds are classified smoothly by their cohomology rings with $\Z_{(2)}$-coefficients, as mentioned at the end of \cite{ch-su:arXiv}, where $\Z_{(2)}$ is the localized ring at $2$.

\section{Cohomology rings and square vanishing elements} \label{sec:cohomology ring and Pontrjagin class of Bmfd}
We recall a Bott tower in \eqref{BTower}, and one can express
\[
\text{$B_j=P(\uC\oplus\gamma^{\alpha_j})$ with $\alpha_j\in H^2(B_{j-1})$,}
\]
where $\uC$ denotes the trivial complex line bundle and $\gamma^{\alpha_j}$ denotes the complex line bundle over $B_{j-1}$ with $\alpha_j$ as the first Chern class for $j=1, \ldots, n$.
Using the Borel-Hirzebruch formula \cite{bo-hi:1958} for the cohomology ring of the projective bundle, we have that $H^\ast(B_j)$ is a free module over $H^\ast(B_{j-1})$ via the map $\pi_j^\ast$ on the two generator $1$ and $x_j$ of degree $0$ and $2$, respectively. The ring structure is determined by the single relation
$$
    x_j^2 = \pi_j^\ast(\alpha_j) x_j,
$$ where $x_j$ is the first Chern class of the tautological line bundle over $B_j$.

Using this formula inductively on $j$ and regarding $H^*(B_j)$ as a graded subring of $H^*(B_n)$ through the projections in \eqref{BTower}, namely, setting $x_i := \pi^\ast_n \circ \cdots \circ \pi^\ast_{i+1}(x_i)$, we see that
$$%\begin{equation} \label{eqn:HBn}
    H^*(B_n)=\Z[x_1, \ldots, x_n]/\langle x_j^2 = \alpha_jx_j\mid j=1,\dots,n\big \rangle,
$$%\end{equation}
where $\alpha_1=0$, and $\alpha_j=\sum_{i=1}^{j-1} A^i_j x_i$ with $A^i_j\in \Z$ for $j=2, \ldots, n$.
Since complex line bundles are classified by their first Chern classes, as is well-known, a Bott tower $B_\bullet$ in \eqref{BTower} is completely determined by the list of integers $A^i_j$ $(1 \leq i < j \leq n)$.
In addition, we note that there is the natural filtration of $H^\ast(B_n)$:
$$%\begin{equation}\label{eqn:filtration}
    H^\ast(B_1) \stackrel{\pi_2^\ast}{\hookrightarrow} H^\ast(B_2) \stackrel{\pi_3^\ast}{\hookrightarrow} \cdots \stackrel{\pi_n^\ast}{\hookrightarrow} H^\ast(B_n).
$$%\end{equation}

%So we may denote the total space $B_n$ of $B_\bullet$ by $M(A)$.
%Let $A$ be an upper triangular matrix of size $n$ with zero diagonals of form
%$$
%    A=\left(
%      \begin{array}{ccccc}
%        0 & A^1_2 & A^1_3 &\cdots & A^1_n \\
%         & 0 & A^2_3 & \cdots & A^2_n \\
%         &  & \ddots & \ddots& \vdots \\
%         &   & &0 &A^{n-1}_n\\
%         &   & & &0
%      \end{array}
%    \right).
%$$

Now, let us consider an element in $H^2(B_n)$ whose square vanishes. Assume that a primitive element $z = ax_j + u$ in $H^2(B_n)$ satisfies $z^2=0$, where $a$ is a non-zero integer and $u$ is a linear combination of $x_i$'s for $i<j$. Then, $z^2 = a^2x_j^2 + 2 a x_j u + u^2 = 0 \in H^\ast(B_n)$; i.e., $2au = - a^2 \alpha_j$ and $u^2 = 0$. This implies that a square vanishing element should be of the form $z = ax_j - \frac{a}{2} \alpha_j$ with $\alpha_j^2=0$. Therefore, a primitive element in $H^2(B_n)$ whose square vanishes is either
$x_j-\frac{1}{2}\alpha_j$ or $2x_j-\alpha_j$
up to sign for some $j$, where $\alpha_j^2=0$ in both cases.
Let $X(B_n)$ be the set of all primitive square vanishing elements of $H^\ast(B_n)$ up to sign. Then, $|X(B_n)|$ is equal to the number of $j$'s satisfying $\alpha_j^2=0$, and, hence, is less than $n$. We say that $B_n$ is \emph{$\Q$-trivial} if its cohomology ring is isomorphic to that of $(\CP^1)^n$ with $\Q$-coefficients as graded rings.

\begin{proposition}\label{proposition:Q-trivial Bott manifold}
$B_n$ is $\Q$-trivial if and only if $\alpha_j^2 = 0$ in $H^*(B_n)$ for all $j=1, \ldots, n$.
\end{proposition}
\begin{proof}
If $\alpha_j^2 =0$, then $( x_j - \frac{\alpha_j}{2})^2=0$ in $H^\ast(B_n; \Q)$ because $x_j^2=\alpha_j x_j$.  Since $ x_j - \frac{\alpha_j}{2}$ for $j=1,\dots,n$ generate $H^*(B_n;\Q)$ as a graded ring, this shows that $B_n$ is $\Q$-trivial. Conversely, if $B_n$ is $\Q$-trivial, there are $n$ primitive elements in $H^2(B_n)$ up to sign whose square vanish, which implies the converse by the above discussion.
\end{proof}

It is known that the strong cohomological rigidity holds for the class of $\Q$-trivial Bott manifolds; namely, we have the following theorem.

\begin{theorem}[Choi-Masuda \cite{ch-ma:arXiv}] \label{thm:SCR_Q_trivial}
Any cohomology ring isomorphism between two $\Q$-trivial Bott manifolds is realizable by a diffeomorphism.
\end{theorem}

Put $t= |X(B_n)|$. A Bott tower $B_\bullet$ is said to be \emph{well-ordered} if $\alpha_j^2 =0$ for $j=1, \ldots, t$, and $\alpha_j^2 \neq 0$ for $j=t+1, \ldots, n$.
\begin{lemma}\label{lemma:making well-ordered}
Every Bott manifold $B_n$ admits a well-ordered Bott tower structure.
\end{lemma}
\begin{proof}
    Consider any Bott tower structure of $B_n$ which is not well-ordered. In other words, there exists at least one $j$ such that $\alpha_j^2 \neq 0$ but $\alpha_{j+1}^2=0$.  Remember that $\alpha_{j+1} = \sum_{i=1}^{j-1} A^i_{j+1} x_i + A^j_{j+1} x_j$. If $A^j_{j+1} \neq 0$ and, as assumed, $\alpha_{j+1}^2=0$, then we get $\alpha_{j+1} = A^{j}_{j+1} x_j - \frac{A^{j}_{j+1}}{2} \alpha_j$ with $\alpha_j^2 =0$, which is a contradiction. Hence, $A^{j}_{j+1} =0$. We can interchange the label $j$ and $j+1$, which proves the lemma by following procedure;
    since $A^j_{j+1}=0$, $\gamma^{\alpha_{j+1}}$ can be regarded as a complex bundle over $B_{j-1}$.
Let $\pi : P(\C \oplus \gamma^{\alpha_{j+1}}) \rightarrow B_{j-1}$ be the corresponding projection. Then,
\[
\xymatrix{
    \pi^\ast B_j \ar[d] \ar[drr]^{\tilde{\pi}}  \ar[rr]^{\cong} & & \ar[d] P(\uC \oplus \gamma^{\alpha_{j+1}})=B_{j+1} \\
    P(\uC \oplus \gamma^{\alpha_{j+1}}) \ar[dr]_{\pi} & & \ar[dl]^{\pi_j}
    P(\uC \oplus \gamma^{\alpha_j}) =B_{j} \\
    & B_{j-1}, &
}\] where $\pi^\ast B_j$ is the pullback of $B_j \to B_{j-1}$ by $\pi$. Then, one can see that $\pi^\ast B_j$ is diffeomorphic to $B_{j+1}$, and it gives another Bott tower structure of $B_n$, which is obtained from $B_\bullet$ by interchanging the $j$ and $j+1$ stages.
\end{proof}
From now on, we only consider Bott manifolds with well-ordered Bott tower structure; namely, we assume that any Bott tower which appears this paper is well-ordered.

Let $B_n'$ be another Bott manifold. Suppose that $H^\ast(B_n)$ and $H^\ast(B_n')$ are isomorphic as graded rings. A graded ring isomorphism $\varphi \colon H^\ast(B_n) \to H^\ast(B_n')$ is said to be $k$-stable if there is a graded ring isomorphism $h_k \colon H^\ast (B_k) \to H^\ast(B_k')$ which makes the diagram
$$
\xymatrix{ H^\ast(B_k) \ar@{^{(}->}[rr]^{\pi_n^\ast \circ \cdots \circ \pi_{k+1}^\ast} \ar[d]^{h_k} && H^\ast(B_n) \ar[d]^\varphi \\
H^\ast(B_k') \ar@{^{(}->}[rr]^{\pi_n'^\ast \circ \cdots \circ \pi_{k+1}'^\ast} &&H^\ast(B_n)
}
$$ commute. We note that $\varphi$ should send elements in $X(B_n)$ to elements in $X(B_n')$ up to sign, and $X(B_n)$ forms an basis of $\pi_n^\ast \circ \cdots \circ \pi_{t+1}^\ast(H^2(B_t))$. It implies that $|X(B_n)|= |X(B_n')|$ (say, $t$), and $\varphi$ is $t$-stable.

\begin{theorem}[Ishida \cite{ishida:pre}] \label{thm;Ishida}
    Let $B_n$ and $B_n'$ be two Bott manifolds. If there is an isomorphism $\varphi \colon H^\ast(B_n) \to H^\ast(B_n')$ which is $(n-1)$-stable, and if $h_{n-1}$ is a realizable by a diffeomorphism between $B_{n-1}$ and $B_{n-1}'$, then so is $\varphi$ by a diffeomorphism betwwen $B_n$ and $B_n'$.
\end{theorem}

\section{Classification of low-stage Bott manifolds}
Note that there is only one one-stage Bott manifold $\CP^1$, and every two-stage Bott manifold is $\Q$-trivial. Hence, by Theorem~\ref{thm:SCR_Q_trivial}, the strong cohomological rigidity holds for one- and two-stage Bott manifolds.

\begin{theorem}\label{thm:SCR_three}
    For any cohomology ring isomorphism $\varphi$ between two three-stage Bott manifolds, there is a diffeomorphism between them, which induces $\varphi$.
\end{theorem}
\begin{proof}
    If three-stage Bott manifolds are $\Q$-trivial, then, by Theorem~\ref{thm:SCR_Q_trivial},  $\varphi$ can be realized by diffeomorphism. Otherwise, namely, they are not $\Q$-trivial, then $\varphi$ should be $2$-stable. Since the strong cohomological rigidity holds for two-stage Bott manifolds, by Theorem~\ref{thm;Ishida}, $\varphi$ is realizable.
\end{proof}

Now, we prepare one lemma for proving the cohomological rigidity for four-stage Bott manifolds.
\begin{lemma}\label{lemma:bundle_change}
Let $B_n = P(\uC\oplus \gamma^\alpha)$ and $B'_n = P(\uC\oplus \gamma^\beta)$ be two projective bundles over an $(n-1)$-stage Bott manifold $B_{n-1}$.
If there exists $u \in H^2(B_{n-1})$ such that $\alpha = \beta - 2u$ and $u(u - \beta) = 0$, then $B_n$ is isomorphic to $B'_n$ as bundles.
\end{lemma}
\begin{proof}
    Note that $P(\uC\oplus \gamma^\beta)$ is isomorphic to $P(\gamma^u \oplus \gamma^{\beta+u})$. The total Chern class of $\gamma^{-u} \oplus \gamma^{\beta-u}$ is $(1-u)(1+\beta-u) = 1 + \beta - 2u + u(u-\beta) = 1 + \alpha$. Hence, $\gamma^{-u} \oplus \gamma^{\beta-u}$ and $\uC\oplus \gamma^\alpha$ are isomorphic by \cite[Theorem 3.1]{ishida:pre}. So are $P(\uC\oplus \gamma^\beta)$ and $P(\uC\oplus \gamma^\alpha)$.
\end{proof}

\begin{theorem}\label{thm:CR_four}
    Let $B_4$ and $B_4'$ be four-stage Bott manifolds. Then, $B_4$ is diffeomorphic to $B_4'$ if and only if $H^\ast(B_4)$ is isomorphic to $H^\ast(B_4')$ as graded rings.
\end{theorem}
\begin{proof}
    Let $\varphi \colon H^\ast(B_4) \to H^\ast(B_4')$ be a graded ring isomorphism.
    If both $B_4$ and $B_4'$ are $\Q$-trivial, then, by Theorem~\ref{thm:SCR_Q_trivial}, $\varphi$ can be realized by diffeomorphism. If $|X(B_4)|=3$, then, combining Theorem~\ref{thm:SCR_three} and Theorem~\ref{thm;Ishida}, $\varphi$ also can be realized. Hence, for the above two cases, $B_4$ and $B_4'$ are diffeomorphic.

    Assume that $|X(B_4)|=2$. We denoted by $y_j, \beta_j$ and $B^i_j$ those elements in $H^\ast(B_4')$ which correspond to $x_j, \alpha_j$ and $A^i_j$ in $H^\ast(B_4)$ for $j=1,\ldots, 4$.  Since $\varphi$ is $2$-stable, $\varphi$ induces a ring isomorphism
    $$ \xymatrix{
        H^\ast(B_4)/ \pi_4^\ast \circ \pi_3^\ast (H^\ast(B_2)) \ar@{=}[d] \ar[r] &  H^\ast(B'_4)/ \pi_4'^\ast \circ \pi_3'^\ast (H^\ast(B'_2)). \ar@{=}[d] \\
        \Z[x_3,x_4]/ \langle x_3^2 =0, x_4^2 = A^3_4 x_3x_4 \rangle & \Z[y_3, y_4] / \langle y_3^2 =0, y_4^2 = B^3_4 y_3 y_4 \rangle}
    $$ Hence, since it preserves the set of primitive square vanishing elements, we conclude $A^3_4$ and $B^3_4$ have the same parity, and $\varphi(x_3)$ is either $\e y_3 + w$, $\e (y_4 - \frac{B^3_4}{2}y_3) +w$ (if $B^3_4$ is even) or $\e (2y_4 - B^3_4 y_3) +w$ (if $B^3_4$ is odd), where $\e = \pm 1$ and $w$ is a linear combination of $y_1$ and $y_2$.

    \textbf{CASE 1 : $\varphi(x_3) = \e y_3 + w$.} Note that $\varphi$ is $3$-stable. Hence, $\varphi$ can be realized by diffeomorphism.

    \textbf{CASE 2 : $\varphi(x_3) = \e (y_4 - \frac{B^3_4}{2}y_3) +w$.} Note that $B^3_4$ (say, $b$) is even. If $b=0$, then we may interchange the third and fourth stages of its Bott tower structure as in Lemma~\ref{lemma:bundle_change}. Hence, $\varphi(x_3)$ would be $3$-stable, and, hence, it can be realized. Suppose that $b\neq 0$. Since $x_3(x_3 - \alpha_3)=0$,
\begin{align*}
    0 &= \varphi(x_3 (x_3 - \alpha_3)) = (\e y_4 - \frac{\e b}{2}y_3 + w) (\e y_4 -\frac{\e b}{2}y_3 + w - \varphi(\alpha_3)) \\
        &=y_4 (y_4 - by_3 + 2 \e w -\e \varphi(\alpha_3)) + \frac{by_3}{4}(by_3 - 4\e w + 2 \e\varphi(\alpha_3)) + w^2 - w\varphi(\alpha_3)
\end{align*}
Because $\varphi(\alpha_3)$ is a linear combination of $y_1$ and $y_2$ and $b \neq 0$, we have that
\begin{align}
y_4 (y_4 - by_3 + 2 \e w -\e \varphi(\alpha_3)) &= 0 \in H^\ast(B'_4), \text { and } \label{eqn:case1-1}\\
\frac{by_3}{2}(\frac{by_3}{2} -2 \e w + \e\varphi(\alpha_3)) &= 0 \in H^\ast(B'_4). \label{eqn:case1-2}
\end{align}
Hence, by \eqref{eqn:case1-1}, $\beta_4 = by_3 - 2\e w + \e \varphi(\alpha_3)$. Let $u = \frac{b y_3}{2}$. Then, by \eqref{eqn:case1-2}, $u(\beta_4 - u)= 0$. Hence, by Lemma~\ref{lemma:bundle_change}, we have an isomorphism $f \colon B'_4 \to P(\uC\oplus \gamma^{\beta-2u})$ as bundles over $B'_3$. This isomorphism gives a new Bott tower structure of $B'_4$ whose 3rd and 4th stages are interchangable. The interchange map is denoted by $g$. The new Bott tower structure obtained by $g \circ f (B'_4)$ is denoted by $B''_{\bullet}$. Note that $f$ and $g$ are diffeomoprhisms, and $B''_\bullet$ is well-ordered. Hence, one can easily check that $g^\ast \circ f^\ast \circ \varphi \colon H^\ast(B_4) \to H^\ast(B''_4)$ is $3$-stable. Therefore, $g^\ast \circ f^\ast \circ \varphi$ is realizable, and, hence, so is $\varphi$.

    \textbf{CASE 3 : $\varphi(x_3) = \e (2y_4 - B^3_4 y_3) +w$.} Note that both $A^3_4$ (say, $a$) and $B^3_4$ (say, $b$) are odd. We may also assume that
    $\varphi^{-1}(y_3) = \ve (2x_4 - a x_3) +z$, where $\ve = \pm1$ and $z$ is a linear combination of $x_1$ and $x_2$.
    Since $x_3(x_3 - \alpha_3)=0$,
\begin{align*}
    0 &= \varphi(x_3 (x_3 - \alpha_3)) = (2\e y_4 - b\e y_3 + w) (2\e y_4 - b\e
    y_3 + w - \varphi(\alpha_3)) \\
        &=4y_4 (y_4 - by_3 +  \e w -\e \frac{\varphi(\alpha_3)}{2}) + b^2y_3(y_3 - \frac{2}{b}\e w +  \frac{1}{b}\e\varphi(\alpha_3)) + w^2 - w\varphi(\alpha_3)
\end{align*}
Hence, $\beta_4 = by_3 - \e w + \e \frac{\varphi(\alpha_3)}{2}$, $\beta_3 = \frac{2\e w }{b} - \frac{\e \varphi(\alpha_3)}{b}$ and $w^2 = w\varphi(\alpha_3)$.
Note that $\beta_3^2 = \frac{1}{b^2} \varphi(\alpha_3^2) \neq 0 \in H^4(B'_4)$. Similarly, we also have $\alpha_3^2 = \frac{1}{a^2} \varphi^{-1}(\beta_3^2)$. Thus, $\alpha_3^2 = \frac{1}{a^2 b^2} \alpha_3^2$. Since $\alpha_3^2$ does not vanish, $a^2 b^2 = 1$. Hence, $|a|=|b|=1$. We may assume that $a=b=1$. Then, $\beta_3 = 2\e w - \e \varphi(\alpha_3)$, and $\beta_4 = y_3  - \e w  + \frac{\e \varphi(\alpha_3)}{2} = y_3 - \frac{\beta_3}{2}$. Similarly, we have $\alpha_4 = x_3 - \frac{\alpha_3}{2}$. By Lemma~\ref{lemma:bundle_change}, we have a bundle isomorphism $f \colon P(\uC\oplus \gamma^{\varphi(\alpha_3)}) \to B'_3$ over $B_2'$. Then, we obtain the pullback $f^\ast B'_4 = P(\uC\oplus \gamma^{y_3 - \frac{\varphi(\alpha_3)}{2}})$ of $B'_4$ by $f$, and we obtain the induced diffeomorphism $\tilde{f} \colon P(\uC\oplus \gamma^{y_3 - \frac{\varphi(\alpha_3)}{2}}) \to B'_4$. On the other hand, since any cohomology ring isomorphism between two Hirzebruch surfaces is realizable, we consider a diffeomorphism $g \colon B'_2 \to B_2$ which induces $\varphi$ restricted by $H^\ast(B_2)$. Then, we also obtain the pullback ${g^{-1}}^\ast (f^\ast B'_4) = P(\uC\oplus \gamma^{x_3 - \frac{\alpha_3}{2}})$ of $f^\ast B'_4$ by $g^{-1}$, and we also have the induced diffeomorphism $\widetilde{g^{-1}} \colon P(\uC\oplus \gamma^{x_3 - \frac{\alpha_3}{2}}) \to f^\ast B'_4$; see the following diagram
$$
\xymatrix{
    P(\uC\oplus \gamma^{x_3 - \frac{\alpha_3}{2}}) \ar[d] \ar[r]^{\widetilde{g^{-1}}} &  P(\uC\oplus \gamma^{y_3 - \frac{\varphi(\alpha_3)}{2}})   \ar[rr]^{\tilde{f}} \ar[d] && P(\uC\oplus \gamma^{y_3 - \frac{\beta_3}{2}}) \ar[d]^{\pi'_4}\\
    P(\uC\oplus \gamma^{\alpha_3}) \ar[d] \ar[r]& P(\uC\oplus \gamma^{\varphi(\alpha_3)}) \ar[rr]^{f} \ar[rd]&& \ar[ld]^{\pi'_3} P(\uC\oplus \gamma^{\beta_3}) \\
    B_2  && \ar[ll]^g B_2'. &
}
$$
Note that $P(\uC\oplus \gamma^{x_3 - \frac{\alpha_3}{2}}) = P(\uC\oplus \gamma^{\alpha_4})$, and, hence, $P(\uC\oplus \gamma^{\alpha_4}) \to P(\uC\oplus \gamma^{\alpha_3}) \to B_2$ is a Bott tower structure of $B_4$. Hence, $\tilde{f} \circ \widetilde{g^{-1}}$ is a diffeomorphism between $B_4$ and $B'_4$.

In the three above cases, we have shown that $B_4$ and $B'_4$ are diffeomorphic, which proves the theorem.
\end{proof}

\begin{example}
Let $B_4$ be a $4$-stage Bott manifold with the Bott tower structure $P(\uC\oplus \gamma^{x_3 - \frac{\alpha_3}{2}}) \to P(\uC\oplus \gamma^{\alpha_3}) \to B_2$. Consider four homomorphisms $\varphi_k \colon H^\ast(B_4) \to H^\ast(B_4)$ $(k=1, \ldots, 4)$ defined by
\begin{enumerate}
  \item $\varphi_1(x_1) = x_1$, $\varphi_1(x_2) =x_2$, $\varphi_1(x_3) = 2x_4 - x_3 + \alpha_3$, and $\varphi_1(x_4) = x_4$;
  \item $\varphi_2(x_1) = x_1$, $\varphi_2(x_2) =x_2$, $\varphi_2(x_3) = 2x_4 - x_3 + \alpha_3$, and $\varphi_2(x_4) = x_4 - x_3 + \frac{\alpha_3}{2}$;
  \item $\varphi_3(x_1) = x_1$, $\varphi_3(x_2) =x_2$, $\varphi_3(x_3) = - 2x_4 + x_3$, and $\varphi_3(x_4) = -x_4$;
  \item $\varphi_4(x_1) = x_1$, $\varphi_4(x_2) =x_2$, $\varphi_4(x_3) = - 2x_4 + x_3$, and $\varphi_4(x_4) = -x_4+x_3-\frac{\alpha_3}{2}$.
\end{enumerate} Then, they are all well-defined, and are graded ring isomorphisms. Moreover, they are all under the third case of the proof of Theorem~\ref{thm:CR_four}.
\end{example}

We remark that a cohomology ring isomorphism $\varphi$ is realizable unless it is under the last case of the proof of Theorem~\ref{thm:CR_four}. However, we do not know whether $\varphi$ of the last case is realizable or not. In order to prove the strong cohomological rigidity for $4$-stage Bott manifolds, what we need is that any automorphism of the cohomology ring of $B_4$ with the Bott tower structure $P(\uC\oplus \gamma^{x_3 - \frac{\alpha_3}{2}}) \to P(\uC\oplus \gamma^{\alpha_3}) \to B_2$ under the last case is realizable. We note that there are only finitely many such automorphisms. Since we may assume that $\varphi(x_1) = x_1$ and $\varphi(x_2)=x_2$, there are only four essential automorphisms $\varphi_k$ ($k=1,\ldots,4$).

\begin{problem}
    Are $\varphi_k$'s ($k=1, \ldots, 4$) realizable?
\end{problem}

\section*{Acknowledgements}
The author would many thank to Professor Matthias Kreck, Housdorff Research Institute for Mathematics, for inviting him to HIM and supporting nice environment to complete this work, and thank to Anna Abczynski, Bonn University, for pointing out many small mistakes and unclear explanations, which helped me improve the paper significantly.

\bigskip
\bibliographystyle{amsplain}
\providecommand{\bysame}{\leavevmode\hbox to3em{\hrulefill}\thinspace}

\end{document}